\documentclass[12pt,reqno]{amsart}
\usepackage{color}
\usepackage{tikz}
\usepackage{url}
\newtheorem{Theorem}{Theorem}[section]

\def\qed{\ifhmode\textqed\fi
	\ifmmode\ifinner\hfill\quad\qedsymbol\else\dispqed\fi\fi}
\def\textqed{\unskip\nobreak\penalty50
	\hskip2em\hbox{}\nobreak\hfill\qedsymbol
	\parfillskip=0pt \finalhyphendemerits=0}
\def\dispqed{\rlap{\qquad\qedsymbol}}

\def\p{\mathfrak{p}}
\def\q{\mathfrak{q}}
\def\m{\mathfrak{m}}
\def\ZZ{\mathbb{Z}}
\def\v{\textup{v}}
\def\Ass{\textup{Ass}}
\def\astab{\textup{astab}}
\def\vstab{\textup{vstab}}
\def\dstab{\textup{dstab}}
\def\rstab{\textup{rstab}}
\def\Min{\textup{Min}}

\def\reg{\textup{reg}}
\def\depth{\textup{depth}}

\begin{document}
	
	\title{Comparison of stability indices of powers of graded ideals}	
	\author{Antonino Ficarra, Emanuele Sgroi}
	
	\address{Antonino Ficarra, BCAM -- Basque Center for Applied Mathematics, Mazarredo 14, 48009 Bilbao, Basque Country -- Spain}
	\address{Ikerbasque, Basque Foundation for Science, Plaza Euskadi 5, 48009 Bilbao, Basque Country -- Spain}
	\email{aficarra@bcamath.org,\,\,\,\,\,\,\,\,\,\,\,\,\,antficarra@unime.it}
	
	\address{Emanuele Sgroi, Department of mathematics and computer sciences, physics and earth sciences, University of Messina, Viale Ferdinando Stagno d'Alcontres 31, 98166 Messina, Italy}
	\email{emasgroi@unime.it}
	
	\thanks{
	}
	
	
	
	\begin{abstract}
		In this paper, we compare the index of ass-stability $\text{astab}(I)$ and the index of $\text{v}$-stability $\text{vstab}(I)$ of powers of a graded ideal $I$. We prove that $\text{astab}(I)=1\le\text{vstab}(I)$ for any graded ideal $I$ in a 2-dimensional polynomial ring, and that $\text{vstab}(I)$ can be any positive integer in this situation. Moreover, given any integers $a,b\ge1$, we construct a graded ideal $I$ in a $3(a+1)$-dimensional polynomial ring such that $(\text{astab}(I),\text{vstab}(I))=(a,b)$.
	\end{abstract}\bigskip
	
	\maketitle
	
	\section*{Introduction}
	Let $S=K[x_1,\dots,x_n]$ be a standard graded polynomial ring over a field $K$, let $I\subset S$ be a graded ideal, and let $\Ass(I)$ be the set of associated primes of $I$. By a famous result of Brodmann \cite{B79} the sequence $\{\Ass(I^k)\}_{k>0}$ stabilizes. That is $\Ass(I^{k+1})=\Ass(I^k)$ for all $k\gg0$. We denote this common sets of associated primes by $\Ass^\infty(I)$. The least integer $k_0$ for which $\Ass(I^k)=\Ass(I^{k_0})$ for all $k\ge k_0$ is called the \textit{index of ass-stability} of $I$ and is denoted by $\astab(I)$.
	
	Let $\p\in\Ass(I)$. The $\v_\p$-number of $I$ is the least degree $\v_\p(I)$ of a homogeneous polynomial $f\in S$ for which $I:f=\p$, \cite{CSTVV20}. Whereas, the $\v$-number of $I$ is defined as
	$$
	\v(I)=\min_{\p\in\Ass(I)}\v_\p(I).
	$$
	
	It was proved by Conca \cite{Conca23}, and independently by the authors of this paper \cite{FS2}, that for each $\p\in\Ass^\infty(I)$, the function $\v_\p(I^k)$ is eventually linear of the form $a_\p k+b_\p$, and also that $\v(I^k)$ is eventually linear of the form $ak+b$, where $b\in\ZZ$ and $a=\alpha(I)=\min\{d:\ I_d\ne0\}$ is the \textit{initial degree} of $I$, see \cite[Theorem 4.1]{FS2}.
	
	Let $\p\in\Ass^\infty(I)$. The least integer $k_0$ such that $\p\in\Ass(I^k)$ and $\v_\p(I^k)=a_\p k+b_\p$ for all $k\ge k_0$, is called the \textit{index of $\v_\p$-stability} of $I$ and is denoted by $\vstab_\p(I)$. The \textit{$\v$-stability index} of $I$, introduced in \cite{BMS24}, is defined as the least integer $\vstab(I)$ for which $\v(I^k)=\alpha(I)k+b$ for all $k\ge\vstab(I)$.
	
	The goal of this paper is to compare the stability indices $\astab(I)$ and $\vstab(I)$ of a graded ideal $I\subset S$. We are partly inspired by the paper \cite{HM} of Herzog and Mafi, where $\astab(I)$ is compared with the index of depth stability of $I$.
	
	For $\p\in\Ass^\infty(I)$, let $\astab_\p(I)=\min\{k_0:\ \p\in\Ass(I^k)\ \textup{for all}\ k\ge k_0\}$. By definition, we have $\vstab_\p(I)\ge\astab_\p(I)$ and $\max_{\p\in\Ass^\infty(I)}\astab_\p(I)=\astab(I)$. Hence
	$$
	\max_{\p\in\Ass^\infty(I)}\vstab_\p(I)\ \ge\ \astab(I).
	$$
	
	Since $\v(I^k)=\min_{\p\in\Ass^\infty(I)}\v_\p(I^k)$ for all $k\ge\astab(I)$, and the minimum of finitely many linear functions is also a linear function, it follows that
	$$
	\vstab(I)\ \le\ \max_{\p\in\Ass^\infty(I)}\vstab_\p(I).
	$$
	
	In view of the above inequalities, one may expect that $\vstab(I)\ge\astab(I)$. In Section \ref{sec2}, we prove that $\text{astab}(I)=1\le\text{vstab}(I)$ for any graded ideal $I\subset S$, provided that $\dim(S)\le2$. In this situation, we show that $\vstab(I)$ can be any given integer. On the other hand, in Section \ref{sec3}, given any integers $a,b\ge1$, we construct a graded ideal $I$ in a $3(a+1)$-dimensional polynomial ring such that $(\text{astab}(I),\text{vstab}(I))=(a,b)$. Hence, in general, the indices $\astab(I)$ and $\vstab(I)$ are not comparable.
	
	\section{The case $\dim(S)\le2$}\label{sec2}
	
	In this section we prove that $\astab(I)\le\vstab(I)$ provided that $\dim(S)\le2$.
	\begin{Theorem}\label{Thm:dim2}
		Let $\dim(S)\le2$. Then $\astab(I)=1$ for any graded ideal $I\subset S$. Moreover, let $S=K[x,y]$ and $I=(x^{2b+1},x^2y^{2b-1},y^{2b+1})$ with $b\ge1$. Then
		\begin{enumerate}
			\item[\textup{(a)}] $\v(I^k)=(2b+1)k$, for $1\le k\le b-1$,
			\item[\textup{(b)}] $\v(I^k)=(2b+1)k-1$, for $k\ge b$,
			\item[\textup{(c)}] $(\astab(I),\vstab(I))=(1,b)$. 
		\end{enumerate}
	\end{Theorem}
	\begin{proof}
		Let $\dim(S)\le2$. By \cite[Remark 1.1]{HM} we have $\astab(I)=1$ for any graded ideal $I\subset S$. Now, let $S=K[x,y]$ and $I=(x^{2b+1},x^2y^{2b-1},y^{2b+1})$. Notice that the statement (c) follows from (a) and (b). So, we prove (a) and (b).\smallskip
		
		(a) Let $\m=(x,y)$, and notice that $I^k$ is a $\m$-primary ideal for all $k\ge1$. Hence $\v(I^k)=\v_\m(I^k)$ for all $k\ge1$. Any $\m$-primary monomial ideal $J\subset S$ is of the form $J=(x^{a_1},x^{a_2}y^{b_2},\dots,x^{a_{m-1}}y^{b_{m-1}},y^{b_m})$ for integer sequences ${\bf a}:a_1>\dots>a_{m-1}>0$ and ${\bf b}:0<b_2<\dots<b_m$, with $m\ge2$. It is shown in \cite[Proposition 5.5]{FS2} that
		\begin{equation}\label{eq:vmI}
			\v_\m(J)\ =\ \min_{1\le i\le m-1}\{a_i+b_{i+1}-2\}.
		\end{equation}
		
		In order to apply this formula in our situation, we first prove that
		\begin{equation}\label{I^k-ij}
			I^k\ =\ (u_{i,j}:\ 0\le i\le k,\,0\le j\le i),
		\end{equation}
		with
		$$
		u_{i,j}\ =\ x^{(2b+1)(k-i)+2(i-j)}y^{(2b+1)i-2(i-j)}.
		$$
		
		To this end, notice that
		\begin{equation}\label{eq:I^k-pqr}
			\begin{aligned}
				I^k\ &=\ ((x^{2b+1})^{p}(x^2y^{2b-1})^{q}(y^{2b+1})^{r}:\ p+q+r=k)\\
				&=\ (x^{(2b+1)p+2q}y^{(2b-1)q+(2b+1)r}:\ p+q+r=k).
			\end{aligned}
		\end{equation}
		
		Write $p=k-i$ for $0\le i\le k$, and set $r=j$. Then $q+j=i$ and so $0\le j\le i$. With these positions, we have $(2b+1)p+2q=(2b+1)(k-i)+2(i-j)$ and
		$$
		(2b-1)q+(2b+1)r=(2b-1)(i-j)+(2b+1)j=(2b+1)i-2(i-j).
		$$
		
		These computations together with formula (\ref{eq:I^k-pqr}) show that equation (\ref{I^k-ij}) holds.
		
		For a monomial $u=x^{a}y^{b}\in S$, let $\deg_x u=a$ and $\deg_y u=b$. Now, let $1\le k<b$. We claim that
		$$
		\deg_x u_{0,0}>\deg_x u_{1,0}>\deg_x u_{1,1}>\deg_x u_{2,0}>\dots>\deg_x u_{k,k-1}>\deg_x u_{k,k}=0
		$$
		and
		$$
		0=\deg_y u_{0,0}<\deg_y u_{1,0}<\deg_y u_{1,1}<\deg_y u_{2,0}<\dots<\deg_y u_{k,k-1}<\deg_y u_{k,k}.
		$$
		
		Since $\deg_x u_{i,j}+\deg_{y}u_{i,j}=(2b+1)k$ for all $i,j$, it is enough to prove the first set of inequalities. It is clear that $\deg_x u_{i,j}>\deg_x u_{i,j+1}$ for $0\le j\le i-1$. It remains to prove that $\deg_x u_{i,i}>\deg_x u_{i+1,0}$, for $0\le i\le k-1$. Suppose for a contradiction that $\deg_x u_{i,i}\le\deg_x u_{i+1,0}$. Then
		$$
		(2b+1)(k-i)\ \le\ (2b+1)(k-(i+1))+2(i+1).
		$$
		
		Since $i+1\le k+1\le b$, we obtain $2b+1\le2(i+1)\le2b$, which is impossible. This proves our claim. Hence, for $1\le k<b$, by applying formula (\ref{eq:vmI}) we obtain
		\begin{align*}
			\v(I^k)\ =&\ \min_{\substack{0\le i\le k-1\\ 0\le j\le i+1}}\{\deg_x u_{i+1,j}+\deg_y u_{i+1,j+1}-2,\,\deg_{x}u_{i,i}+\deg_yu_{i+1,0}-2\}\\
			=&\ \min_{0\le i\le k-1}\{(2b+1)k,\,(2b+1)(k+1)-2(i+1)-2\}\\[8pt]
			=&\ (2b+1)k,
		\end{align*}
		as desired.\smallskip
		
		(b) Let $w=x^{2b}y^{(2b+1)(k-1)}$ with $k\ge b$. We have $\deg(w)=\alpha(I^k)-1$ and so $w\notin I^k$. Moreover $xw=(x^{2b+1})y^{(2b+1)(k-1)}\in I^k$ and $yw=(x^2y^{2b-1})^by^{(2b+1)(k-b)}\in I^k$. Hence $I^k:w=\m$ and so $\v(I^k)\le\alpha(I)k-1$ for all $k\ge b$. Since by~\cite[Proposition 2.2]{F2023} we have $\v(I^k)\ge\alpha(I)k-1$ for all $k\ge1$, the assertion follows.
	\end{proof}
	
	\section{The case $\dim(S)\ge3$}\label{sec3}
	
	We now come to the main result of this paper.
	\begin{Theorem}
		Let $a,b\ge1$ be positive integers. Then, there exist a number $N>0$ and a monomial ideal $I\subset S=K[x_1,\dots,x_N]$ such that
		$$
		(\astab(I),\vstab(I))=(a,b).
		$$
	\end{Theorem}
	\begin{proof}
		If $a=1$, the statement follows from Theorem \ref{Thm:dim2}.
		
		Now, let $a\ge2$. Let $J_i=(x_{3i-2}x_{3i-1},x_{3i-2}x_{3i},x_{3i-1}x_{3i})$ for $1\le i\le a-1$ and set $J=J_1+\dots+J_{a-1}$. Notice that each $J_i$ is an edge ideal with linear resolution. Thus, \cite[Theorem 5.1]{F2023} implies that $\v(J_i^k)=2k-1$ for all $k\ge1$, and in particular $\vstab(J_i)=1$ for $1\le i\le a-1$. Since $\alpha(J_i)=2$ and the generators of the ideals $J_i$ are monomials in pairwise disjoint sets of variables, \cite[Theorem 5.2]{FM} implies that
		\begin{equation}\label{eq:Jv}
			\begin{aligned}
				\v(J^k)\ &=\ (\min_{1\le i\le a-1}\alpha(J_i))k+(\sum_{i=1}^{a-1}\alpha(J_i)-\min_{1\le i\le a-1}\alpha(J_i)-(a-1))\\
				&=\ 2k+a-3,
			\end{aligned}
		\end{equation}
		for all $k\ge1$. In particular $\vstab(J)=1$.
		
		It is immediate to see that $$\Ass(J_i)=\Min(J_i)=\{(x_{3i-2},x_{3i-1}),(x_{3i-2},x_{3i}),(x_{3i-1},x_{3i})\}.$$
		
		Since $J_i^2:(x_{3i-2}x_{3i-1}x_{3i})=(x_{3i-2},x_{3i-1},x_{3i})=\m_i$ and edge ideals satisfy the persistence property (see \cite[Theorem 2.15]{MMV}) we have $\Ass(J_i^k)=\Ass(J_i)\cup\{\m_i\}$ for all $k\ge2$. Hence $\astab(J_i)=2$ for $1\le i\le a-1$. If $a=2$, then $J=J_1$ and $\astab(J)=a=2$, as wanted. Suppose now $a>2$. By \cite[Theorem 4.1(3)]{NT} we have
		$$
		\Ass((J_1+J_2)^k)\ =\ \{\p+\q\ :\ 0\le\ell<k,\, \p\in\Ass(J_1^{k-\ell}),\, \q\in\Ass(J_2^{\ell+1})\}.
		$$
		
		From this formula we see immediately that $\Ass(J_1+J_2)\subseteq\Ass((J_1+J_2)^2)\subseteq\Ass((J_1+J_2)^3)\subseteq\cdots$ and $\astab(J_1+J_2)=3$. Applying again this result to $J_1+J_2$ and $J_3$ and proceeding iteratively in this way we obtain that $\Ass(J)\subseteq\Ass(J^2)\subseteq\Ass(J^3)\subseteq\cdots$ and $\astab(J)=a$.
		
		Now, we consider the ideal
		$$
		I\ =\ x_0^{2b-1}J+(x^{2b+1},x^2y^{2b-1},y^{2b+1})
		$$
		in the polynomial ring $S=K[x_0,x_1,\dots,x_{3a},x,y]$ in $3(a+1)$ variables.
		
		We claim that $(\astab(I),\vstab(I))=(a,b)$. To this end, we set $L=x_0^{2b-1}J$ and $H=(x^{2b+1},x^2y^{2b-1},y^{2b+1})$. Using \cite[Corollary 2.3]{FM} and formula (\ref{eq:Jv}) we have
		\begin{equation}\label{eq:vL}
			\begin{aligned}
				\v(L^k)\ &=\ \min\{\v((x_0^{(2b-1)k}))+\alpha(J)k,\,\alpha((x_0)^{2b-1})k+\v(J^k)\}\\
				&=\ \min\{(2b-1)k-1+2k,(2b-1)k+2k+a-3\}\\
				&=\ (2b+1)k-1,
			\end{aligned}
		\end{equation}
		for all $k\ge1$, where we used that $a\ge2$.
		
		By \cite[Lemma 2.1]{FM} we have $\Ass(L^k)=\Ass(J^k)\cup\{(x_0)\}$ for all $k\ge1$. Hence $\Ass(L)\subseteq\Ass(L^2)\subseteq\Ass(L^3)\subseteq\cdots$. Using this fact and since $\Ass(H^k)=\{(x,y)\}$ for all $k\ge1$, it follows from \cite[Theorem 4.1(3)]{NT} that
		\begin{equation}\label{eq:assI}
			\Ass(I^k)\ =\ \bigcup_{\substack{0\le\ell<k\\ \p\in\Ass(L^{k-\ell})}}\{(\p,x,y)\}\ =\ \{(\p,x,y)\ :\ \p\in\Ass(L^k)\}.
		\end{equation}
		
		Since $\astab(L)=\astab(J)=a$, this formula implies that $\astab(I)=a$, too.
		
		Now, since $\Ass(H^k)=\{(x,y)\}$ for all $k\ge1$, by using \cite[Corollary 4.2]{FM} and the formulas (\ref{eq:vL}) and (\ref{eq:assI}), we obtain that
		\begin{equation}\label{eq:vI-LH}
			\begin{aligned}
				\v(I^k)\ &=\ \min_{\substack{0\le\ell<k\\ \p\in\Ass(L^{k-\ell})}}\{\v_\p(L^{k-\ell})+\v_{(x,y)}(H^{\ell+1})\}\\
				&=\ \min_{0\le\ell<k}\{(2b+1)(k-\ell)-1+\v_{(x,y)}(H^{\ell+1})\}.
			\end{aligned}
		\end{equation}
		
		Using Theorem \ref{Thm:dim2}(a)-(b) we get
		$$
		(2b+1)(k-\ell)-1+\v_{(x,y)}(H^{\ell+1})\ =\ \begin{cases}
			(2b+1)k+2b&\textup{for}\ 1\le\ell\le b-2,\\
			(2b+1)k+2b-1&\textup{for}\ \ell\ge b-1.
		\end{cases}
		$$
		
		Combined with (\ref{eq:vI-LH}) we obtain that $\v(I^k)=(2b+1)k+2b$ for $1\le k<b$, and 
		$\v(I^k)=(2b+1)k+2b-1$ for $k\ge b$. Hence $\vstab(I)=b$, as desired.
	\end{proof}\smallskip
	
	Let $I\subset S$ be a graded ideal. It is known that $\depth\,S/I^{k+1}=\depth\,S/I^k$ for all $k\gg0$. The least integer $k_0$ such that $\depth\,S/I^{k+1}=\depth\,S/I^k$ for all $k\ge k_0$ is called the index of depth stability of $I$ and is denoted by $\dstab(I)$. Likewise, it is also well-known that there exists $k_*>0$ for which the Castelnuovo-Mumford regularity $\reg(I^k)$ is equal to a linear function for all $k\ge k_*$. The index of regularity stability $\rstab(I)$ of $I$ is defined as the least integer with this property.\smallskip
	
	At the moment (see \cite[Question 5.6]{FSPackA}) we do not known whether for given integers $a,d,r,v\ge1$, there exists a graded ideal $I\subset S=K[x_1,\dots,x_N]$ such that
	$$
	(\astab(I),\dstab(I),\rstab(I),\vstab(I))\ =\ (a,d,r,v).
	$$\smallskip
	
	\textbf{Acknowledgment.} A. Ficarra was partly supported by the Grant JDC2023-051705-I funded by
	MICIU/AEI/10.13039/501100011033 and by the FSE+.
	
\end{document}